\documentclass[sfdefaults=false]{scrartcl}

\usepackage[british]{babel}
\usepackage[utf8]{inputenc}
\usepackage[T1]{fontenc}
\usepackage[cal=boondox,bb=boondox]{mathalfa}
\usepackage{euler}
\usepackage{amsmath}
\usepackage{amssymb}
\usepackage{amsthm}
\usepackage{palatino}
\usepackage{enumitem}
\usepackage{etoolbox}
\usepackage{tikz-cd}
\usetikzlibrary{decorations.pathmorphing}
\usepackage{dsfont}
\usepackage{calc}
\usepackage{mathtools}
\usepackage{xstring}
\usepackage{pict2e,picture} 
\usepackage[style=alphabetic,backend=biber,opcittracker=true]{biblatex}
\DeclareFieldFormat{eprint:eudml}{%
  \href{http://eudml.org/doc/#1}{EuDML : #1}%
}
\DeclareFieldFormat{eprint:jstor}{%
  \href{http://www.jstor.org/stable/#1}{JSTOR : #1}%
}
\DeclareFieldFormat{eprint:euclid}{%
  \href{http://projecteuclid.org/euclid.cmp/#1}{Project Euclid : #1}%
}
\usepackage{hyperref}
\usepackage[nameinlink]{cleveref}

\DeclareMathAlphabet{\mathfrak}{U}{jkpmia}{m}{it}
\SetMathAlphabet{\mathfrak}{bold}{U}{jkpmia}{bx}{it}

\newcounter{descriptcount}

\newlist{enumdescript}{description}{1}
\setlist[enumdescript,1]{%
  before={\setcounter{descriptcount}{0}%
          \renewcommand*\thedescriptcount{\arabic{descriptcount}}},
        font={\bfseries\stepcounter{descriptcount}\thedescriptcount.~}
}

\crefname{paragraph}{paragraph}{paragraphs}
\Crefname{Paragraph}{Paragraph}{Paragraphs}
\crefname{subsubsubappendix}{subsubsection}{subsubsections}
\Crefname{subsubsubappendix}{Subsubsection}{Subsubsections}
\crefname{subsubsubsubappendix}{paragraph}{paragraphs}
\Crefname{subsubsubsubappendix}{Paragraph}{Paragraphs}

\theoremstyle{definition}
\newtheorem{definition}[subsubsection]{\definitionautorefname}
\newcommand{\definitionautorefname}{Definition}

\newtheorem{construct}[subsubsection]{\constructautorefname}
\newcommand{\constructautorefname}{Construction}

\newcommand{\conjautorefname}{Conjecture}

\theoremstyle{remark}
\newtheorem{remark}[subsubsection]{\remarkautorefname}
\newcommand{\remarkautorefname}{Remark}

\newtheorem{exmp}[subsubsection]{\exmpautorefname}
\newcommand{\exmpautorefname}{Example}

\newtheorem{warning}[subsubsection]{\warningautorefname}
\newcommand{\warningautorefname}{Warning}

\newcommand{\notationautorefname}{Notation}

\theoremstyle{plain} \newtheorem{thm}[subsubsection]{\thmautorefname}
\newcommand{\thmautorefname}{Theorem}

\newtheorem{corlr}[subsubsection]{\corlrautorefname}
\newcommand{\corlrautorefname}{Corollary}

\newcommand{\prorautorefname}{Property}

\newtheorem{pros}[subsubsection]{\prosautorefname}
\newcommand{\prosautorefname}{Proposition}

\newtheorem{lemma}[subsubsection]{\lemmaautorefname}
\newcommand{\lemmaautorefname}{Lemma}

\newcommand{\scholiumautorefname}{Scholium}

\makeatletter
\DeclareRobustCommand\widecheck[1]{{\mathpalette\@widecheck{#1}}}
\def\@widecheck#1#2{%
  \setbox\z@\hbox{\m@th$#1#2$}%
  \setbox\tw@\hbox{\m@th$#1%
    \widehat{%
      \vrule\@width\z@\@height\ht\z@
      \vrule\@height\z@\@width\wd\z@}$}%
  \dp\tw@-\ht\z@
  \@tempdima\ht\z@ \advance\@tempdima2\ht\tw@ \divide\@tempdima\thr@@
  \setbox\tw@\hbox{%
    \raise\@tempdima\hbox{\scalebox{1}[-1]{\lower\@tempdima\box
        \tw@}}}%
  {\ooalign{\box\tw@ \cr \box\z@}}}
\makeatother

\makeatletter
\DeclareRobustCommand{\bbDelta}{{\mathpalette\bb@Delta\relax}}
\newcommand{\bb@Delta}[2]{%
  \begingroup
  \sbox\z@{$\m@th#1\Delta$}%
  \dimendef\Dht=6 \dimendef\Dwd=8
  \setlength{\Dwd}{\wd\z@}%
  \setlength{\Dht}{\ht\z@}%
  \begin{picture}(\Dwd,\Dht)
  \put(0,0){$\m@th#1\Delta$}
  \put(.42\Dwd,.7\Dht){\line(10,-26){.25\Dht}}
  \end{picture}%
  \endgroup
}
\makeatother
\makeatletter
\DeclareRobustCommand{\bbGamma}{{\mathpalette\bb@Gamma\relax}}
\newcommand{\bb@Gamma}[2]{%
  \begingroup
  \sbox\z@{$\m@th#1\Gamma$}%
  \dimendef\Dht=6 \dimendef\Dwd=8
  \setlength{\Dwd}{\wd\z@}%
  \setlength{\Dht}{\ht\z@}%
  \begin{picture}(\Dwd,\Dht)
  \put(0,0){$\m@th#1\Gamma$}
  \put(.47\Dwd,.025\Dht){\line(0,1){.95\Dht}}
  \end{picture}%
  \endgroup
}
\makeatother

\newcommand\mapstofonc{\mathrel{\ooalign{$\rightsquigarrow$\cr%
  \kern-.105ex\raise.325ex\hbox{\scalebox{1}[0.388]{$\mid$}}\cr}}}

\makeatletter
\newcommand{\addchar}[2]{%
  \@tfor\letter:=#1\do{%
    \letter#2
  }%
}
\makeatother

\newcommand{\Eta}{H} %
\newcommand{\cat}[1]{\mathfrak{#1}}
\newcommand{\op}[1]{{#1}^{\mathrm{op}}}
\newcommand{\func}[1]{\mathcal{\addchar{#1}{\!}}\,}

\newcommand{\yo}{\func{y}}
\newcommand{\oprd}[1]{\mathscr{#1}}

\newcommand{\lmamalg}[1]{\mathbin{\mathop{\amalg}\limits_{#1}}}

\newcommand{\infgrpds}{\cat{\infty\textnormal{-}Grpd}}
\newcommand{\inflcats}[1]{\cat{\mathnormal{(\infty,#1)}\textnormal{-}Cat}}
\newcommand{\infcats}{\inflcats{1}}
\newcommand{\infinfcats}{\inflcats{\omega}}
\newcommand{\stromcat}{\cat{Str}\omega\textnormal{-}\cat{Cat}}
\newcommand{\infzcats}{\inflcats{\mathbb{Z}}}
\newcommand{\catsp}{\cat{CatSp}}

\newcommand{\funcs}[2]{\cat{Fun}\bigl(\cat{#1},\cat{#2}\bigr)}

\newcommand{\inrt}{\text{\textnormal{inrt}}}

\newcommand{\elem}{\text{\textnormal{el}}}
\newcommand{\flagged}{\text{\textnormal{fl}}}
\newcommand{\stabd}{\text{\textnormal{st}}}
\newcommand{\biptd}{\text{\textnormal{bpt}}}
\newcommand{\walkeq}{\cat{eq}}
\newcommand{\algs}[1]{{#1}\cat{\textnormal{-}Alg}}
\DeclareMathOperator{\obj}{obj}

\DeclareMathOperator*{\colim}{colim}
\DeclareMathOperator{\cells}{cell}

\DeclarePairedDelimiter{\characmap}{\ulcorner}{\urcorner}

\title{Categorical spectra as\\ pointed $(\infty,\mathbb{Z})$-categories} \author{David Kern}

\begin{document}

\maketitle{}

\begin{abstract}
  Lessard's $\mathbb{Z}$-categories are an analogue of
  $\omega$-categories possessing cells in all positive \emph{and
    negative} dimensions. Categorical spectra, developed by Stefanich,
  are an analogue of spectra obtained by replacing the suspension of
  pointed $\infty$-groupoids by that of pointed
  $(\infty,\omega)$-categories.

  We give an $\infty$-categorical definition of weak
  $\mathbb{Z}$-categories (alias $(\infty,\mathbb{Z})$-categories),
  and show categorical spectra to be equivalent to pointed
  $(\infty,\mathbb{Z})$-categories. In particular, we show that the
  stable cells of categorical spectra coincide with the natural cells
  of $(\infty,\mathbb{Z})$-categories, and recover Lessard's
  description of spectra as pointed weak $\mathbb{Z}$-groupoids.
\end{abstract}

\tableofcontents{}

\section{Introduction}
\label{sec:introduction}

Several important constructions of higher categories exhibit certain
stability properties with respect to the shifting of categorical
level. For example:
\begin{itemize}
\item For $\cat{C}$ a $(\infty,1)$-category with pullbacks,
  \cite{haugseng18:_iterat} constructs an $(\infty,n+1)$-category
  $\cat{Span}_{n}(\cat{C})$ of $n$-iterated
  spans. 
  Each $(\infty,n)$-category $\cat{Span}_{n}(\cat{C})$ is canonically
  pointed at the terminal object $\ast$ of $\cat{C}$, it is shown
  in~\cite[Proposition 8.3]{haugseng18:_iterat} that for any
  $n\in\mathbb{N}$ we have
  $\Omega_{\ast}\cat{Span}_{n+1}(\cat{C})
  \coloneqq\hom_{\cat{Span}_{n+1}(\cat{C})}(\ast,\ast)
  \simeq\cat{Span}_{n}(\cat{C})$.
\item The Morita $(\infty,n+1)$-category $\cat{Mor}_{n}(\cat{V})$ of a
  symmetric monoidal $(\infty,1)$-category $\cat{V}$, constructed
  in~\cite{haugseng17:_morit}, with as objects the
  $\oprd{E}_{n}$-algebras in $\cat{V}$, $1$-morphisms the
  $\oprd{E}_{n-1}$-algebras in $\oprd{E}_{n}$-bimodules, and so
  on and so forth, is canonically pointed at the monoidal unit $I$
  with its trivial algebra structure. By~\cite[Corollary
  5.52]{haugseng17:_morit}, for any $n\in\mathbb{N}$ we have
  $\Omega_{(I,\cdot_{\text{triv}})}\cat{Mor}_{n+1}(\cat{V})
  \simeq\cat{Mor}_{n}(\cat{V})$.
\item For $\cat{M}$ a symmetric monoidal presentable stable
  $(\infty,1)$-category, \cite[Notation 5.2.1]{stefanich20:_presen}
  inductively defines a stable $(\infty,n+1)$-category
  $\cat{Mod}^{n}(\cat{M})$ of $\cat{Mod}^{n-1}(\cat{M})$-modules in
  presentable stable $(\infty,n)$-categories. As observed
  in~\cite[Example 13.3.6]{stefanich21:_higher_quasic_sheav}, for any
  $n\in\mathbb{Z}$ we have
  $\Omega_{\cat{Mod}^{n}(\cat{M})}\cat{Mod}^{n+1}(\cat{M})
  \simeq\cat{Mod}^{n}(\cat{M})$.
\end{itemize}
These properties were formalised by~\cite[Remark
13.2.15]{stefanich21:_higher_quasic_sheav} as manifestations of the
fact that each of these construction assemble, letting
$\ell\in\mathbb{N}$ vary over all possible values, into a categorical
spectrum, a variant of spectra where the shifting of homotopical
degree for pointed $\infty$-groupoids is replaced by the shifting of
categorical degree for pointed $(\infty,\omega)$-categories.

More precisely, while for pointed $\infty$-groupoids we have an
adjunction
$\Sigma\colon\infgrpds_{\ast}\rightleftarrows\infgrpds_{\ast}\colon\Omega$
where $\Sigma$ sends $n$-truncated $\infty$-groupoids (aka
$n$-groupoids) to $(n+1)$-groupoids and $\Omega$ sends $n$-groupoids
to $(n-1)$-groupoids, so too for $(\infty,\omega)$-categories do we
have an adjunction
$\Sigma\colon\infinfcats_{\ast}\rightleftarrows\infinfcats_{\ast}\colon\Omega$
in which $\Sigma$ sends $(\infty,n)$-categories to
$(\infty,n+1)$-categories and $\Omega$ sends $(\infty,n)$-categories
to $(\infty,n-1)$-categories. We can then mimic the definition of the
$\infty$-category of spectra as the limit of the tower
$\cdots\xrightarrow{\Omega}\infgrpds_{\ast}\xrightarrow{\Omega}\infgrpds_{\ast}$
by replacing $\infgrpds$ by $\infinfcats$ to define the
$\infty$-category of categorical spectra.

In particular, whereas only grouplike $\oprd{E}_{\infty}$-monoids
can be delooped to (connective) spectra, any
$\oprd{E}_{\infty}$-monoidal $(\infty,\omega)$-category can be
delooped to a connective categorical spectrum, and this is in fact an
equivalence of categories. Thus, if the connective categorical spectra
have an incarnation as more purely algebraic objects, one may wonder
if the same is true for general categorical spectra.

Some intuition, or evidence, for this idea is the fact that
categorical spectra can be seen to be made up of $k$-cells, for
$k\in\mathbb{Z}$ any \emph{relative} integer, which behave in the same
way as the cells of a higher category (in that they have a source and
target, and compose along those).

This idea of cells of any dimension can also be found in Lessard's
$\mathbb{Z}$-categories, defined and studied
in~\cite{lessard19:_spect_local_finit_z_group}
and~\cite{lessard22:_categ_i} with the goal of relating them to
spectra. Their definition follows Kan's construction of spectra by
performing the stabilisation procedure on $\omega$-categories
\emph{before} passing to pointed objects.

Lessard's work is written is strict and model-categorical language,
which has the result of obscuring the border between homotopical and
algebraic arguments. In this note, we give an
``algebraic''\footnote{In the sense of homotopical algebra, not of the
  classical algebraic definitions of higher categories.}
$\infty$-categorical definition of weak $\mathbb{Z}$-categories,
henceforth referred to as $(\infty,\mathbb{Z})$-categories. We then
extend Lessard's equivalence between pointed groupoidal weak
$\mathbb{Z}$-categories and spectra by showing that pointed
$(\infty,\mathbb{Z})$-categories are equivalent to categorical spectra
(\cref{thm:main-thm}), and prove a delooping result by extending this
description to $\oprd{E}_{k}$-monoidal
$(\infty,\mathbb{Z})$-categories for any $k$. We finally show that the
stable cells of categorical spectra can be interpreted as actual cells
in the corresponding $(\infty,\mathbb{Z})$-category
(\cref{thm:cells}), from which we recover Lessard's characterisation
of spectra as pointed $(\infty,\mathbb{Z})$-groupoids.

\subsection*{Open problem: extension to (Gray) cubical suspension}

We mention a possible avenue of research opened by our results.

There are two operations on $(\infty,\omega)$-categories that can be
interpreted as a ``suspension''. The first is the globular suspension
considered in this article, and the second is the Gray cylinder, the
(lax) Gray tensor product with the $1$-globe, a more cubical
suspension. The cubes were used by~\cite{campion23:_gray} to
characterise the Gray tensor product of $(\infty,\omega)$-categories,
thanks to the density results of~\cite{campion22:_cubes}.

In~\cite{masuda24:_algeb_categ_spect}, a Gray tensor product for
categorical spectra is constructed. Given our results, one would
expect that this tensor product should correspond to the ``smash
product'' (on pointed objects) for a Gray tensor product on
$(\infty,\mathbb{Z})$-categories. It is conjectured in~\cite[Remark
2.3]{campion23:_gray} that the cubes might provide a dense
presentation for $(\infty,\omega)$-categories. If this were the case,
it would then be possible to reproduce the constructions of this note
replacing the suspension of globular sums by the Gray cylinder on
cubes, and leverage the methods of~\cite{campion23:_gray} to obtain
the desired tensor product.

\subsection*{Acknowledgements}

If the numerous references throughout the text do not make it clear, I
wish to emphasise how much this note relies on~\cite[Chapters
12--14]{lessard19:_spect_local_finit_z_group} and~\cite[Chapter
13]{stefanich21:_higher_quasic_sheav}, and indeed can be seen as
simply an attempt to bridge the two works.

The author was supported by the Göran Gustafsson Foundation for
Research in Natural Sciences and Medicine.

\section{A homotopy-algebraic definition of
  $(\infty,\mathbb{Z})$-categories}
\label{sec:an-algebr-defin}

\subsection{Recollections on weak $\omega$-categories and strict
  $\mathbb{Z}$-categories}
\label{sec:recoll-weak-omega}

\begin{itemize}
\item The category $\Theta$ is a full subcategory of $\stromcat$ on
  the strict $\omega$-categories which are free on globular sums. The
  functors which are free on morphisms of $\omega$-graphs between
  globular sums will be called \textbf{inert}; they are the right
  class of an orthogonal factorisation system whose left class is
  called \textbf{active}.
\item For any $n\in\mathbb{N}$, there is a distinguished object
  denoted $\cat{D}_{n}$, the $n$-disk (or $n$-globe, or walking
  $n$-cell). These objects are said to be \textbf{elementary}. The
  subcategory $\Theta^{\elem}\subset\Theta$ on the elementary objects
  is isomorphic to the (non-reflexive) globe category $\mathbb{G}$
  represented as the generating graph
  \begin{equation}
    \label{eq:glob-cat}
    \begin{tikzcd}
      \overline{0} \arrow[r,shift left,"i_{0}^{+}"] \arrow[r,shift
      right,"i_{0}^{-}"'] & \overline{1} \arrow[r,shift
      left,"i_{1}^{+}"] \arrow[r,shift right,"i_{1}^{-}"'] & \cdots
      \arrow[r,shift left,"i_{n-1}^{+}"] \arrow[r,shift
      right,"i_{n-1}^{-}"'] & \overline{n} \arrow[r,shift
      left,"i_{n}^{+}"] \arrow[r,shift right,"i_{n}^{-}"'] &
      \cdots\text{,}
    \end{tikzcd}
  \end{equation}
  with the relations
  $i_{n+1}^{+}i_{n}^{\varepsilon}=i_{n+1}^{-}i_{n}^{\varepsilon}$ for
  any $n\in\mathbb{N}$ and any $\varepsilon\in\{+,-\}$.
\item The category $\Theta$ is closed under the operation of
  \textbf{suspension} $\Xi\colon\stromcat\to\stromcat$, which takes a
  strict $\omega$-category $\cat{C}$ and returns the $\omega$-category
  $\Xi\cat{C}$ with two objects and hom $\omega$-category $\cat{C}$
  from the first to the second. The subcategory $\mathbb{G}$ is
  further closed under this endofunctor $\Xi$, with
  $\Xi\cat{D}_{n}=\cat{D}_{n+1}$.

  As noted for example in~\cite{campion23:_gray}, the functor $\Xi$
  factors through the category $\Theta_{\ast,\ast}$ of bipointed
  objects, that is objects under $\ast\amalg\ast$; we denote
  $\Xi_{\biptd}\colon\Theta\to\Theta_{\ast,\ast}$ this
  ``enhanced'' suspension.
\end{itemize}

The choice of the elementary objects and the factorisation system
$(\text{inert},\text{active})$ on $\op{\Theta}$ is known
from~\cite{chu21:_homot_segal} as a structure of \textbf{algebraic
  pattern}, and allows one to talk of Segal objects. Following the
terminology, for any $T\in\op{\Theta}$ we write
${\op{\Theta}}^{\elem}_{T/}$ for the category of \emph{inert}
morphisms to an elementary object in $\op{\Theta}$

\begin{definition}
  A \textbf{flagged $(\infty,\omega)$-category} is a presheaf
  $\func{X}\colon\op{\Theta}\to\infgrpds$ on $\Theta$ which satisfies
  the Segal condition:
  \begin{equation}
    \label{eq:segal-omega-cat}
    \forall
    T\in\op{\Theta},\func{X}(T)\xrightarrow{\simeq}
    \varprojlim_{E\in{\op{\Theta}}^{\elem}_{T/}}\func{X}(E)\text{.}
  \end{equation}
\end{definition}

\begin{remark}
  \label{remark:globular-pres-segcond}
  As shown in~\cite[Lemma
  3.3.3]{lessard19:_spect_local_finit_z_group}, any object
  $T\in\Theta$ admits a so-called \textbf{globular presentation}
  \begin{equation}
    \label{eq:glob-pres-state}
    T\simeq\cat{D}_{n_{1}}\lmamalg{\cat{D}_{m_{1}}}
    \dots\lmamalg{\cat{D}_{m_{p-1}}}\cat{D}_{n_{p}}\text{,}
  \end{equation}
  for a uniquely determined list of natural integers
  $n_{1},\cdots,n_{p},m_{1},\cdots,m_{p-1}\in\mathbb{N}$ with
  $m_{i}<n_{i},n_{i+1}$, as a gluing of globes along their
  boundaries. Furthermore, by~\cite[Lemme 2.3.22]{ara10:_sur_groth}
  this diagram provides an initial\footnote{It is only stated
    in~\cite{ara10:_sur_groth} that the inclusion is initial as a
    functor of $1$-categories, that is that its slices are (non-empty
    and) connective, but in fact the proof shows that these slices are
    zigzags hence contractible, whence initiality as a functor of
    $(\infty,1)$-categories.} subcategory of
  ${\op{\Theta}}^{\elem}_{T/}$, so that the Segal condition takes the
  more concrete form
  \begin{equation}
    \label{eq:segal-omega-cat-decomp}
    \func{X}(T)\xrightarrow{\simeq}
    \func{X}(\overline{n_{1}})\times_{\func{X}(\overline{m_{1}})}
    \cdots
    \times_{\func{X}(\overline{m_{p-1}})}\func{X}(\overline{m_{p}})\text{.}
  \end{equation}
\end{remark}

Another way of writing the general Segal conditions is that the
presheaf $\func{X}$ is local with respect to the morphisms
$\yo(T)\to\colim\yo(E)$. This formulation as local objects is useful
for defining univalence-completeness (also known as Rezk-completeness)
in a similar manner, to restrict from flagged
$(\infty,\omega)$-categories to actual $(\infty,\omega)$-categories.

The \textbf{walking equivalence}, as a presheaf on $\Theta$, is
defined as the gluing
\begin{equation}
  \label{eq:walkeq-def}
  \walkeq=\colim\left(\begin{tikzcd}[cramped]
      & {\yo[1]} & & {\yo[3]} & & {\yo[1]} & \\
      {\yo[0]} \arrow[from=ur] & & {\yo[2]} \arrow[from=ul]
      \arrow[from=ur] & & {\yo[2]}
      \arrow[from=ul] \arrow[from=ur] & & {\yo[0]} \arrow[from=ul]
    \end{tikzcd}\right)
\end{equation}
of two degenerate $2$-simplices along their non-degenerate edge,
expressing an arrow with a left and a right inverses.

\begin{definition}
  A Segal presheaf $\func{X}\colon\op{\Theta}\to\infgrpds$ is
  \textbf{univalent-complete} if it is furthermore local with respect
  to the arrows $\Xi^{n}(\walkeq\to\ast),n\geq0$.

  An \textbf{$(\infty,\omega)$-category} is a flagged
  $(\infty,\omega)$-category which is univalent-complete.

  We denote $\infinfcats^{\flagged}\supset\infinfcats$ the
  $(\infty,1)$ categories of flagged $(\infty,\omega)$-categories and
  of $(\infty,\omega)$-categories respectively.
\end{definition}

\begin{definition}
  \label{def:theta-stab}
  The \textbf{stable cell category} is the sequential colimit of
  $1$-categories
  \begin{equation}
    \label{eq:theta-stab-def}
    \Theta_{\mathbb{Z}}\coloneqq\colim\bigl(
    \Theta\xrightarrow{\Xi}\Theta\xrightarrow{\Xi}
    \Theta\xrightarrow{\Xi}\cdots\bigr)\text{.}
  \end{equation}
\end{definition}

\begin{construct}[Shifts and infinite suspensions]
  \label{construct:shifts-theta-stab}
  As outlined in~\cite{lessard19:_spect_local_finit_z_group}, we have
  shifting operators on the category $\Theta_{\mathbb{Z}}$.
  \begin{itemize}
  \item The inclusion of the $k$th factor, for $k\in\mathbb{N}$, in
    the tower~\cref{eq:theta-stab-def} gives a functor
    $\Xi^{\infty-k}\colon\Theta\to\Theta_{\mathbb{Z}}$. Furthermore,
    we have $\Xi^{\infty-k}\circ\Xi\simeq\Xi^{\infty-(k-1)}$, so that
    for $\ell<0$ we may also define
    $\Xi^{\infty-\ell}\coloneqq\Xi^{\infty}\circ\Xi^{-\ell}$, where
    $\Xi^{\infty}\coloneqq\Xi^{\infty-0}$.
  \item By~\cite[Remark 13.2.7]{stefanich21:_higher_quasic_sheav} the
    endofunctor $\Xi\colon\Theta\to\Theta$ induces an invertible
    endofunctor of $\Theta_{\mathbb{Z}}$, which we also denote
    $\Xi$. More explicitly, it comes from the morphism of
    ($\mathbb{N}$-indexed) diagrams
    \begin{equation}
      \label{eq:morph-diagra-shift}
      \begin{tikzcd}
        \Theta \arrow[r,"\Xi"] \arrow[d,"\Xi"'] & \Theta
        \arrow[r,"\Xi"] \arrow[d,"\Xi"] & \Theta
        \arrow[r,"\Xi"] \arrow[d,"\Xi"] & \cdots \\
        \Theta \arrow[r,"\Xi"'] & \Theta \arrow[r,"\Xi"'] & \Theta
        \arrow[r,"\Xi"'] & \cdots
      \end{tikzcd}
    \end{equation}
    both having the same colimit $\Theta_{\mathbb{Z}}$. We will also
    denote $\Eta=\Xi^{-1}$ the inverse of this shifting functor.

    By construction, these shifts are compatible with the infinite
    suspensions (and in particular, the suspensions of globular sums)
    in that we have
    $\Xi\circ\Xi^{\infty}\simeq\Xi^{\infty+1}\simeq\Xi^{\infty}\circ\Xi$
    and $\Eta\circ\Xi^{\infty}\simeq\Xi^{\infty-1}$.
  \end{itemize}
\end{construct}

\begin{construct}[The stable globe category]
  \label{construct:stable-globe-cat}
  Since the elementary subcategory $\mathbb{G}\subset\Theta$ is stable
  under $\Xi$, we can also consider
  \begin{equation}
    \label{eq:def-globe-stabd}
    \mathbb{G}_{\mathbb{Z}}\coloneqq\colim\bigl(
    \mathbb{G}\xrightarrow{\Xi}\mathbb{G}\xrightarrow{\Xi}
    \mathbb{G}\xrightarrow{\Xi}\cdots\bigr)\text{.}
  \end{equation}

  By~\cite[Proposition 12.3.2]{lessard19:_spect_local_finit_z_group},
  this category admits an explicit
  description\footnote{\cite{lessard19:_spect_local_finit_z_group}
    actually describes the reflexive stable globe category; however
    in~\cite{lessard22:_categ_i} he switches to the non-reflexive
    one.}, as equivalent to a category $\mathbb{G}_{\stabd}$ obtained
  from the graph
  \begin{equation}
    \label{eq:stabd-glob-cat}
    \begin{tikzcd}
      \cdots \arrow[r,shift left,"i_{-m-1}^{+}"] \arrow[r,shift
      right,"i_{-m-1}^{-}"'] & \overline{-m} \arrow[r,shift
      left,"i_{-m}^{+}"] \arrow[r,shift right,"i_{-m}^{-}"'] & \cdots
      \arrow[r,shift left,"i_{-2}^{+}"] \arrow[r,shift
      right,"i_{-2}^{-}"'] & \overline{-1} \arrow[r,shift
      left,"i_{-1}^{+}"] \arrow[r,shift right,"i_{-1}^{-}"'] &
      \overline{0} \arrow[r,shift left,"i_{0}^{+}"] \arrow[r,shift
      right,"i_{0}^{-}"'] & \overline{1} \arrow[r,shift
      left,"i_{1}^{+}"] \arrow[r,shift right,"i_{1}^{-}"'] & \cdots
      \arrow[r,shift left,"i_{n-1}^{+}"] \arrow[r,shift
      right,"i_{n-1}^{-}"'] & \overline{n} \arrow[r,shift
      left,"i_{n}^{+}"] \arrow[r,shift right,"i_{n}^{-}"'] & \cdots
    \end{tikzcd}
  \end{equation}
  by imposing the relations
  $i_{n+1}^{+}i_{n}^{\varepsilon}=i_{n+1}^{-}i_{n}^{\varepsilon}$ for
  any $n\in\mathbb{Z}$ and any $\varepsilon\in\{+,-\}$.

  By analogy with the unstable case, for any $\ell\in\mathbb{Z}$ we
  denote $\Xi^{\infty}\cat{D}_{\ell}$ the image of $\overline{\ell}$
  under the embedding
  $\mathbb{G}_{\mathbb{Z}}\hookrightarrow\Theta_{\mathbb{Z}}$ (even
  though the ones with $\ell<0$ are not in the image of the functor
  $\Xi^{\infty}$), so that we always have
  $\Xi^{\infty}\cat{D}_{\ell}=\Xi^{\infty+\ell}\cat{D}_{0}$.
\end{construct}

\begin{lemma}[{\cite[Lemma
    12.4.1]{lessard19:_spect_local_finit_z_group}}]
  \label{lemma:stabl-globular-pres}
  For any $T\in\Theta_{\mathbb{Z}}$, there is a uniquely determined
  list of relative integers
  $n_{1},\cdots,n_{p},m_{1},\cdots,m_{p-1}\in\mathbb{Z}$ with
  $m_{i}<n_{i},n_{i+1}$ such that
  \begin{equation}
    \label{eq:stab-glob-pres}
    T\simeq\Xi^{\infty}\cat{D}_{n_{1}}
    \lmamalg{\Xi^{\infty}\cat{D}_{m_{1}}}
    \dots\lmamalg{\Xi^{\infty}\cat{D}_{m_{p-1}}}
    \Xi^{\infty}\cat{D}_{n_{p}}\text{.}
  \end{equation}
\end{lemma}

\begin{definition}[{\cite[Definition
    12.4.2]{lessard19:_spect_local_finit_z_group}}]
  A \textbf{strict $\mathbb{Z}$-category} is a functor
  $\func{X}\colon\op{\Theta_{\mathbb{Z}}}\to\cat{Set}$ taking stable
  globular presentations to limits, that is such that for any
  $T\in\Theta_{\mathbb{Z}}$ with globular presentation as
  in~\cref{eq:stab-glob-pres}, the function
  as~\cref{eq:segal-omega-cat-decomp} is an isomorphism.
\end{definition}

\begin{remark}
  It is shown in~\cite{lessard22:_categ_i} that, much as strict
  $\omega$-categories can be seen as globular sets (presheaves on
  $\mathbb{G}$) equipped with a structure of algebra over the ``free
  $\omega$-category'' monad, so too are strict $\mathbb{Z}$-categories
  also presheaves on $\mathbb{G}_{\stabd}$ equipped with a structure
  of algebra over the colimit monad. We will not pursue this point of
  view further in this note.
\end{remark}

\subsection{Flagged $(\infty,\mathbb{Z})$-categories}
\label{sec:flagged-z-categories}

Our first step for putting the algebraic definition of strict
$\mathbb{Z}$-categories on the same footing as that of
$(\infty,\omega)$-categories will be to endow the category
$\Theta_{\mathbb{Z}}$ with a structure of algebraic pattern.

\begin{lemma}
  \label{lemma:biptd-suspend-loop}
  The functor $\Xi_{\biptd}\colon\Theta\to\Theta_{\ast,\ast}$ admits a
  right-adjoint $\Eta_{\biptd}\colon\Theta_{\ast,\ast}\to\Theta$.
\end{lemma}

\begin{proof}
  The right adjoint $\Eta_{\biptd}$ is constructed by sending a bipointed
  $\omega$-category $(\cat{C};C_{0},C_{1})$ to the hom
  $\omega$-category $\cat{C}(C_{0},C_{1})$. That it is indeed
  right-adjoint to $\Xi$ is~\cite[Lemma 1.8]{campion22:_cubes}.
\end{proof}

\begin{warning}
  \label{warning-desuspend-infcats-notcell}
  The functor $\Xi\colon\stromcat\to\stromcat$ (forgetting the
  bipointing of a suspension) also admits a right-adjoint
  $\Eta\colon\stromcat\to\stromcat$, which sends a strict
  $\omega$-category $\cat{C}$ to
  $\colim_{x,y\in\obj(\cat{C})}\hom(x,y)$. However, it does not
  restrict to an endofunctor of $\Theta$ since the desuspension of a
  globular sum need not be connected (which globular sums always are).
\end{warning}

\begin{pros}
  The functor $\op{\Xi}\colon\op{\Theta}\to\op{\Theta}$ defines a
  morphism of algebraic patterns.
\end{pros}

\begin{proof}
  That $\Xi$ preserves elementary objects is immediate since
  $\Xi\overline{n}=\overline{n+1}$. To see that it also preserves
  inert morphisms, it is enough to observe that the restriction of
  $\Xi$ to the inert subcategory of $\Theta$, which is the category of
  pasting shapes, coincides with the natural operation of suspension
  on pasting shapes (for example in their sylvan description,
  \emph{cf.}~\cref{construct:inrt-stable-trees}).
  
  For the preservation of active morphisms, we will use their
  characterisation as left-orthogonal (in $\Theta$, rather than
  $\op{\Theta}$) to the inerts. Let then $f\colon A\to B$ be an active
  morphism, and consider a lifting problem
  \begin{equation}
    \label{eq:lift-pb-shift-active}
    \begin{tikzcd}
      \Xi A \arrow[d,"\Xi f"'] \arrow[r] & C \arrow[d,tail,"e"] \\
      \Xi B \arrow[r] & D
    \end{tikzcd}
  \end{equation}
  with $e\colon C\to D$ inert.

  As $\Xi f$ is in the image of $\Xi$ it is in particular a bipointed
  map, fitting in the extended diagram as below-left
  \begin{equation}
    \label{eq:lift-pb-bipt-adj}
    \begin{tikzcd}[row sep=tiny]
      & \Xi A \arrow[dd,"\Xi f"'] \arrow[r] & C \arrow[dd,"e"] \\
      \ast\amalg\ast \arrow[ur] \arrow[dr] & & \\
      & \Xi B \arrow[r] & D
    \end{tikzcd}\qquad\qquad\qquad
    \begin{tikzcd}
      A \arrow[d,"f"'] \arrow[r] & \Eta_{\biptd} C \arrow[d,tail,"H e"] \\
      B \arrow[r] & \Eta_{\biptd} D
    \end{tikzcd}
  \end{equation}
  which also displays $C$ and $D$ as bipointed objects and $e$ as a
  bipointed map. Thus the square can be adjointed
  by~\cref{lemma:biptd-suspend-loop} to the diagram as above-right.

  But as $\Eta_{\biptd}$ is as easily seen as $\Xi$ to preserve inert maps, the
  latter square is a lifting problem with an active map on the left
  and an inert one on the right, so admits an essentially unique
  solution, which can be adjointed to an equally unique solution of
  the lifting problem~\cref{eq:lift-pb-shift-active}. Thus $\Xi f$ is
  active.
\end{proof}

Recall that by~\cite[Corollary 5.5]{chu21:_homot_segal}, the category
of algebraic patterns has limits and filtered colimits, which are preserved by the forgetful functor to
$\infty$-categories.

\begin{definition}
  The algebraic pattern $\op{\Theta}_{\mathbb{Z}}$ is the sequential
  colimit of the tower of algebraic patterns
  \begin{equation}
    \label{eq:tower-suspend-theta}
    \op{\Theta}\xrightarrow{\op{\Xi}}\op{\Theta}\xrightarrow{\op{\Xi}}
    \op{\Theta}\xrightarrow{\op{\Xi}}\cdots\text{.}
  \end{equation}

  We call \textbf{flagged $(\infty,\mathbb{Z})$-categories} the Segal
  $\op{\Theta}_{\mathbb{Z}}$-objects in $\infgrpds$, and denote
  $\infzcats^{\flagged}$ their $\infty$-category.
\end{definition}

\begin{remark}
  \label{remark:theta-st-ulying-1cat}
  The maps in the underlying tower of $1$-categories
  $\op{\Theta}\xrightarrow{\op{\Xi}}
  \op{\Theta}\xrightarrow{\op{\Xi}}\cdots$ are injective on objects
  and morphisms (so that they will be cofibrations in any reasonable
  model category for $(\infty,1)$-categories), so the colimit in
  $\infcats$ coincides with the strict sequential colimit of
  $1$-categories, which is the category ${\op{\Theta}_{\mathbb{Z}}}$
  defined in~\cref{def:theta-stab}.

  In addition, the elementary subcategory
  ${\op{\Theta}_{\mathbb{Z}}}^{\elem}$ is likewise computed as the
  colimit of
  $\op{\mathbb{G}}\xrightarrow{\op{\Xi}}\op{\mathbb{G}}
  \xrightarrow{\op{\Xi}}\cdots$ so that it also recovers the category
  $\op{\mathbb{G}_{\mathbb{Z}}}$ of~\cref{construct:stable-globe-cat}.
\end{remark}

As in~\cref{remark:globular-pres-segcond}, we can use the
decompositions of~\cref{lemma:stabl-globular-pres} to obtain an
explicit form for the Segal conditions. For this we will also need a
more explicit description of the inert subcategory of
$\op{\Theta_{\mathbb{Z}}}$.

\begin{construct}
  \label{construct:inrt-stable-trees}
  Recall that the inert subcategory of $\Theta$ can be seen as a
  category of trees, as for example
  in~\cite[\S{}2.3.2]{ara10:_sur_groth}. These are functors from
  $\func{T}\colon\op{\omega}\to\cat{LinOrdSet}$, where $\omega$ is the
  ordered set $\{0\leq1\leq2\leq\cdots\}$, satisfying the conditions
  that $\func{T}(0)=\ast$ and there is some $i\gg0$ such that
  $\func{T}(i)=\emptyset$. The pasting shape corresponding to a tree
  $\func{T}$ has as $k$-cells the sectors between branches at level
  $k$, where the left-to-right ordering of sectors gives the
  composition order for sequences of cells.

  A presheaf $\func{T}$ satisfying the finiteness condition but not
  $\func{T}(i)=\ast$ may be known as a \textbf{forest}. The
  endomorphism $\operatorname{succ}\colon\omega\to\omega$ induces a
  desuspension operator $\Eta=\operatorname{succ}^{\ast}$ on forests,
  but as in~\cref{warning-desuspend-infcats-notcell} it does not
  restrict to an endomorphism on trees. Its left-adjoint
  $\operatorname{succ}_{!}$ does restrict to trees, and models the
  suspension functor $\Xi$.

  We can likewise define a \textbf{stable tree} to be a functor
  $\func{T}\colon\op{\mathbb{Z}}\to\cat{LinOrdSet}$ such that
  $\func{T}(i)=\emptyset$ for some $i$ and there is some $j<i$ such
  that for all $j^{\prime}\leq j$, $\func{T}(j^{\prime})=\ast$. We
  denote $\Theta^{\inrt}_{\stabd}$ the category of stable trees. The
  mutually inverse operators
  $\operatorname{succ},\operatorname{prec}
  \colon\mathbb{Z}\to\mathbb{Z}$ induce automorphisms
  $\operatorname{succ}^{\ast}=\operatorname{prec}_{!},
  \operatorname{prec}^{\ast}=\operatorname{succ}_{!}$ of
  $\Theta^{\inrt}_{\stabd}$.
\end{construct}

\begin{pros}
  \label{pros:theta-z-inrt-tree}
  The category $\Theta_{\stabd}^{\inrt}$ is equivalent to
  $\Theta_{\mathbb{Z}}^{\inrt}
  \simeq\colim\bigl(\Theta^{\inrt}\xrightarrow{\Xi}\Theta^{\inrt}
  \xrightarrow{\Xi}\cdots\bigr)$.
\end{pros}

\begin{proof}
  The idea is the same as that of the proof of~\cite[Proposition
  12.3.2]{lessard19:_spect_local_finit_z_group}, though some more care
  is required in handling trees rather than globes.
  
  Consider the inclusion
  $\func{i}\colon\op{\omega}\hookrightarrow\op{\mathbb{Z}}$. It
  induces by left Kan extension an ``infinite suspension'' functor
  $\func{i}_{!}\colon\Theta^{\inrt}\to\Theta_{\stabd}^{\inrt}$:
  explicitly, for any tree $\func{T}$, the stable tree
  $\func{i}_{!}\func{T}$ coincides with $\func{T}$ on positive cells
  (since $\func{i}$ is fully faithful) and has value
  $\func{T}(i)=\ast$ for $i\leq0$. Since the successor operation on
  $\mathbb{Z}$ (and thus also, when applicable, the predecessor
  operation) is compatible with the one on $\omega$, the functors
  $\operatorname{succ}_{!}^{n}\circ\func{i}$ define a cocone, and thus
  induce a unique compatible functor
  $\upsilon\colon\Theta^{\inrt}_{\mathbb{Z}}\to\Theta^{\inrt}_{\stabd}$.

  Now we also construct a functor $\gamma$ going the other way. For
  $\func{T}\in\Theta^{\inrt}_{\stabd}$, take $-i\in\mathbb{Z}_{\leq0}$
  such that $\func{T}(j)\simeq\ast$ for all $j\leq-i$ (existing thanks
  to the stable tree condition); it follows by direct examination that
  $\operatorname{succ}_{!}^{i}\func{T}$ is $\func{i}_{!}$ of an
  unstable tree. We then set
  $\gamma(\func{T})
  =\Xi^{\infty-i}(\func{i}_{!}^{-1}\operatorname{succ}_{!}^{i}\func{T})$.
  Note that if we had picked some $-j<-i$ we would still have
  \begin{equation}
    \label{eq:2}
    \Xi^{\infty-j}(\func{i}_{!}^{-1}\operatorname{succ}_{!}^{j}\func{T})
    \simeq\Xi^{\infty-i}\Xi^{i-j}(
    \func{i}_{!}^{-1}\operatorname{succ}_{!}^{j}\func{T})
    \simeq\Xi^{\infty-i}(\func{i}_{!}^{-1}\operatorname{succ}_{!}^{i}\func{T})
  \end{equation}
  so the assignment $\gamma$ is well-defined on objects, and it also
  follows that it is functorial since we can always suspend as much as
  needed. Eventually, as $\upsilon$ and $\gamma$ are both compatible
  with the suspension operators, it follows formally (thanks to the
  universal property of $\Theta^{\inrt}_{\mathbb{Z}}$) that they are
  inverses.
\end{proof}

\begin{corlr}
  \label{corlr:stabl-glob-pres-final}
  Let $T\in\op{\Theta}_{\mathbb{Z}}$ be determined by the integers
  $n_{1},\dots,n_{k},m_{1},\dots,m_{k-1}\in\mathbb{Z}$ according
  to~\cref{lemma:stabl-globular-pres}. The diagram of stable globes
  \begin{equation}
    \label{eq:final-glob-pres-stabl}
    \begin{tikzcd}
      \overline{n_{1}} \arrow[dr] & & \overline{n_{2}} \arrow[dl]
      \arrow[dr] & \cdots & \overline{n_{k-1}} \arrow[dl] \arrow[dr] &
      & \overline{n_{k}}
      \arrow[dl] \\
      & \overline{m_{1}} & & \cdots & & \overline{m_{k-1}} &
    \end{tikzcd}
  \end{equation}
  is initial in $(\op{\Theta}_{\mathbb{Z}})^{\elem}_{T/}$.
\end{corlr}

\begin{proof}
  Thanks to~\cref{pros:theta-z-inrt-tree}, the proof method
  of~\cite[Lemme 2.3.22]{ara10:_sur_groth} (which
  gave~\cref{remark:globular-pres-segcond}) applies again.
\end{proof}

It follows that the Segal condition for $\op{\Theta_{\mathbb{Z}}}$
takes a form as in~\cref{eq:stab-glob-pres}, and in particular that
Segal $\op{\Theta_{\mathbb{Z}}}$-sets are exactly strict
$\mathbb{Z}$-categories.

\begin{lemma}
  The functors of~\cref{construct:shifts-theta-stab} define Segal
  morphisms of algebraic patterns.
\end{lemma}

\begin{proof}
  That they are morphisms of algebraic patterns is immediate from the
  fact that they are canonical morphisms induced by a colimit in the
  $\infty$-category of algebraic patterns.

  The functors $\Xi$ and $\Eta$ are Segal because they are invertible
  between identical pattern structures. Pulling back a
  $\op{\Theta_{\mathbb{Z}}}$-object along a $\Xi^{\infty-k}$ is simply
  restricting it to one component, so again Segality is clear.
\end{proof}

Note that it is not clear that $\Xi^{\infty-k}$ is extendable, that is
satisfies the hypotheses of~\cite[Proposition
7.13]{chu21:_homot_segal}, so that the left-adjoint
$\Xi^{\infty-k}_{!}$ to $\Xi^{\infty-k,\ast}$ needs to be postcomposed
with the ``Segalification'' localisation. Since the only thing we will
need about $\Xi^{\infty-k}_{!}$ is its left-adjoint property, we will
simply omit this localisation from the notation.

\begin{definition}
  \label{def:shifting-ops-flagged-zcats}
  We call
  $\Xi^{\infty-k}_{!}\colon\infinfcats^{\flagged}\to\infzcats^{\flagged}$
  the $k$-shifted \textbf{infinite suspension} and
  $\Eta^{\infty-k}_{!}\coloneqq\Xi^{\infty-k,\ast}\colon\infzcats^{\flagged}
  \to\infinfcats^{\flagged}$ the $k$-shifted \textbf{infinite
    desuspension}.

  We will refer to $\Xi_{!}\simeq\Eta^{\ast}$ and
  $\Eta_{!}\simeq\Xi^{\ast}$ as the \textbf{shifting} (and
  unshifting) operations.
\end{definition}

\subsection{Stable univalence}
\label{sec:stable-univalence}

Now we also want to move the functor
$\op{\Xi}\colon\op{\Theta}\to\op{\Theta}$ out of the patterns and on
to the realm of Segal objects. For this we simply observe the
following:

\begin{lemma}
  The morphism of algebraic patterns
  $\op{\Xi}\colon\op{\Theta}\to\op{\Theta}$ is a strong Segal
  morphism.
\end{lemma}

\begin{proof}
  We need to check that for every $T\in\Theta$, the induced functor
  $\Theta^{\elem}_{/T}\to\Theta^{\elem}_{/\Xi T}$ is final. Recall
  that the globular presentation
  \begin{equation}
    \label{eq:globular-pres}
    T=\overline{n_{1}}\lmamalg{\overline{m_{1}}}\overline{n_{2}}
    \lmamalg{\overline{m_{2}}}\cdots
    \lmamalg{\overline{m_{k-1}}}\overline{n_{k}}
  \end{equation}
  provides a final subcategory of $\Theta^{\elem}_{/T}$, and
  likewise the globular presentation of $\Xi T$ is
  \begin{equation}
    \label{eq:globular-pres-shift}
    \Xi T=\overline{n_{1}+1}\lmamalg{\overline{m_{1}+1}}\overline{n_{2}+1}
    \lmamalg{\overline{m_{2}+1}}\cdots
    \lmamalg{\overline{m_{k-1}+1}}\overline{n_{k}+1}\text{.}
  \end{equation}
  One then sees immediately that $\Xi$ sends the globular presentation
  of $T$ to that of $\Xi T$, so that the induced functor is indeed
  final.
\end{proof}

\begin{remark}
  It follows that $\Xi$ induces a functor
  $\Xi^{\ast}\colon\infinfcats^{\flagged}\to\infinfcats^{\flagged}$. If
  $\cat{C}$ is a flagged $(\infty,\omega)$-category, then
  $\Xi^{\ast}\cat{C}$ evaluates on a globe $\overline{n}$ to
  $\cat{C}_{n+1}$, the type of $(n+1)$-cells of $\cat{C}$. We can thus
  think of $\Xi^{\ast}$ as a desuspension operator for flagged
  $(\infty,\omega)$-categories.

  It admits a left-adjoint $\Xi_{!}$, given by Yoneda extension, that
  is left Kan extension of Segal presheaves.
\end{remark}

\begin{corlr}
  \label{corlr:flaggd-zcats-lim-suspensions}
  The $(\infty,1)$-category $\infzcats^{\flagged}$ is the limit of the
  tower of $(\infty,1)$-categories
  \begin{equation}
    \label{eq:tower-infinfcats-desuspend}
    \cdots\xrightarrow{\Xi^{\ast}}\infinfcats^{\flagged}
    \xrightarrow{\Xi^{\ast}}\infinfcats^{\flagged}
    \xrightarrow{\Xi^{\ast}}\infinfcats^{\flagged}\text{.}
  \end{equation}
\end{corlr}

\begin{proof}
  At the level of presheaf $\infty$-categories, we have as usual that
  \begin{equation}
    \label{eq:limit-preshf-cat-theta}
    \funcs{\op{\Theta}_{\mathbb{Z}}}{\infgrpds}
    =\funcs{\colim_{\mathbb{N}}\op{\Theta}}{\infgrpds}
    \simeq\lim_{\mathbb{N}}\funcs{\op{\Theta}}{\infgrpds}\text{.}
  \end{equation}
  It then only remains to observe, as in~\cite{chu21:_homot_segal},
  that since the Segal condition is a limit condition, this
  equivalence restricts to Segal presheaves.
\end{proof}

Thus a flagged $(\infty,\mathbb{Z})$-category is given by a sequence
$(\cat{C}_{k})_{k\in\mathbb{N}}$ of flagged
$(\infty,\omega)$-categories equipped with equivalences $\cat{C}_{k}\simeq\Xi^{\ast}C_{k+1}$.

\begin{remark}
  The shifting operators of~\cref{def:shifting-ops-flagged-zcats} can
  now be identified with ones defined purely in terms
  of~\cref{eq:tower-infinfcats-desuspend}: the shifting operator
  $\Xi_{!}$ and $\Eta_{!}$ come directly from the periodicity of the
  diagram in the same way that $\Xi$ and $\Eta$ did (so we have
  explicit formulae
  \begin{equation}
    \label{eq:shift-explicit}
    \Xi_{!}((\cat{C}_{k})_{k\in\mathbb{N}})
    =(\cat{C}_{k+1})_{k\in\mathbb{N}}\text{,}
  \end{equation} and
  \begin{equation}
    \label{eq:unshift-explicit}
    \Eta_{!}((\cat{C}_{k})_{k\in\mathbb{N}})
    =(\Xi^{\ast}\cat{C}_{k})_{k\in\mathbb{N}}
    =(\cat{C}_{k-1})_{k\in\mathbb{N}}
  \end{equation}
  with $\cat{C}_{-1}\coloneqq\Xi^{\ast}\cat{C}_{0}$), and more
  importantly the $k$-shifted infinite desuspension functor
  $\Eta^{\infty-k}_{!}\colon\infzcats^{\flagged}\to\infinfcats^{\flagged}$
  is the projection on the $k$th factor.
\end{remark}

\begin{lemma}
  If a flagged $(\infty,\omega)$-category $\cat{C}$ is
  univalent-complete, then so is its desuspension $\Xi^{\ast}\cat{C}$.
\end{lemma}

\begin{proof}
  We need to check that $\Xi^{\ast}\cat{C}$ is local with respect to
  the functors $\Xi_{!}^{n}\walkeq\to\Xi_{!}^{n}\ast=\cat{D}_{n}$, for
  all $n\geq0$, that is that the induced functor
  $\hom(\cat{D}_{n},\Xi^{\ast}\cat{C})
  \to\hom(\Xi_{!}^{n}\walkeq,\Xi^{\ast}\cat{C})$ is an equivalence. By
  adjunction, this map is equivalent to
  \begin{equation}
    \label{eq:adjunct-walkeq-local}
    \hom(\Xi_{!}\cat{D}_{n},\cat{C})=\hom(\cat{D}_{n+1},\cat{C})
    \to\hom(\Xi_{!}^{n+1}\walkeq,\cat{C})
  \end{equation}
  which is invertible because $\cat{C}$ is univalent-complete.
\end{proof}

\begin{warning}
  \label{warning:univalence-not-stable}
  The converse fails to be true, because univalence-completeness of
  $\Xi^{\ast}\cat{C}$ only implies that $\cat{C}$ is local with
  respect to the functors $\Xi_{!}^{n}\walkeq\to\cat{D}_{n}$ for
  $n\geq1$, but does not recover the case $n=0$.
\end{warning}

Thus the restriction of $\Xi^{\ast}$ to univalent-complete Segal
$\op{\Theta}$-objects defines a functor
$\Xi^{\ast}\colon\infinfcats\to\infinfcats$.

\begin{lemma}
  Let $\cat{I}$ be an $(\infty,1)$-category, let
  $\func{D},\func{D}^{\prime}\colon\cat{I}\to\infcats$ be two
  $\cat{I}$-indexed diagrams of $(\infty,1)$-categories, and let
  $\alpha\colon\func{D}\Rightarrow\func{D}^{\prime}$ be a
  transformation of diagrams which is component-wise fully
  faithful. Then the induced functor
  $\varprojlim_{\cat{I}}\eta
  \colon\varprojlim_{\cat{I}}\func{D}
  \to\varprojlim_{\cat{I}}\func{D}^{\prime}$
  is also fully faithful.
\end{lemma}

\begin{proof}
  An $(\infty,1)$-functor is fully faithful if and only if it is
  right-orthogonal to the class of functors that are essentailly
  surjective on objects (\emph{cf.}~\cite[Proposition
  2.2.1.75]{loubaton24:_categ_theor_categ}). Let us thus consider a
  lifting problem as below-left
  \begin{equation}
    \label{eq:lift-pb-eso-limcat}
    \begin{tikzcd}
      \cat{A} \arrow[d,"{\func{q}}"'] \arrow[r] &
      \varprojlim_{\cat{I}}\func{D}
      \arrow[d,"{\varprojlim_{\cat{I}}\eta}"] \\
      \cat{B} \arrow[r] & \varprojlim_{\cat{I}}\func{D}^{\prime}
    \end{tikzcd}\qquad\qquad\begin{tikzcd}
      \func{const}\cat{A} \arrow[d,"{\func{const}(\func{q})}"']
      \arrow[r] & \func{D}
      \arrow[d,"{\eta}"] \\
      \func{const}\cat{B} \arrow[r] \arrow[ur,dashed] &
      \func{D}^{\prime}
    \end{tikzcd}
  \end{equation}
  where $\func{q}$ is essentially surjective. Since the limit functor
  is right-adjoint to the constant diagram functor $\func{const}$,
  this square corresponds to a lifting problem of diagrams as
  above-right.

  But $\func{const}(\func{q})$ is essentially surjective while $\eta$
  is fully faithful, so this square admits an essentially unique
  dashed lift, which can be adjointed to a(n essentially unique)
  solution to the original lifting problem.
\end{proof}

\begin{definition}
  The $\infty$-category of \textbf{$(\infty,\mathbb{Z})$-categories}
  is the essential image of the functor
  $\lim\bigl(\cdots\xrightarrow{\Xi^{\ast}}\infinfcats
  \xrightarrow{\Xi^{\ast}}\infinfcats\bigr)
  \hookrightarrow\infzcats^{\flagged}$ (induced by the inclusions
  $\infinfcats\hookrightarrow\infinfcats^{\flagged}$).
\end{definition}

Let $\cat{C}$ be a flagged $(\infty,\mathbb{Z})$-category, seen as a
sequence $(\cat{C}_{0},\cat{C}_{1},\dots)$ of flagged
$(\infty,\omega)$-categories with
$\cat{C}_{i}\simeq\Xi^{\ast}\cat{C}_{i+1}$ for all
$i\in\mathbb{N}$. Then $\cat{C}$ is an $(\infty,\mathbb{Z})$-category
if and only if each $\cat{C}_{i}$ is univalent complete.

\begin{definition}
  A flagged $(\infty,\mathbb{Z})$-category is \textbf{stably
    univalent-complete} if it is local with respect to the maps
  $\Xi^{\infty+n}_{!}\walkeq\to\Xi^{\infty+n}\ast$ for all
  $n\in\mathbb{Z}$.
\end{definition}

\begin{remark}
  In contrast to~\cref{warning:univalence-not-stable}, a flagged
  $(\infty,\mathbb{Z})$-category $\cat{C}$ is an
  $(\infty,\mathbb{Z})$-category if and only if its desuspension
  $\Eta_{!}\cat{C}$ is so (if and only if its suspension
  $\Xi_{!}\cat{C}$ is), and the same is true for stable univalence.
\end{remark}

\begin{pros}
  A flagged $(\infty,\mathbb{Z})$-category $\cat{C}$ is an
  $(\infty,\mathbb{Z})$-category if and only if it is stably
  univalent-complete.
\end{pros}

\begin{proof}
  Suppose first that $\cat{C}$ is stably univalent-complete; in
  particular for any $n,k\geq0$ the map
  $\hom(\Xi^{\infty+n-k}_{!}\ast,\cat{C})
  \to\hom(\Xi^{\infty+n-k}_{!}\walkeq,\cat{C})$ is an equivalence. This
  map is equivalent to
  $\hom(\Xi^{n}\ast,\Eta^{\infty-k}_{!}\cat{C})
  \to\hom(\Xi^{n}\ast,\Eta^{\infty-k}_{!}\cat{C})$. Since
  $\Eta^{\infty-k}_{!}\cat{C}$ is $\cat{C}_{k}$, this shows that $\cat{C}$
  is an $(\infty,\mathbb{Z})$-category.

  Now suppose that $\cat{C}$ is an $(\infty,\mathbb{Z})$-category, and
  let $n\in\mathbb{\mathbb{Z}}$. If $n\geq0$, reversing the above
  reasoning shows that $\cat{C}$ is local for the map
  $\Xi^{\infty+n}_{!}\walkeq\to\Xi^{\infty+n}\ast$. If $n<0$, pick
  $k\geq-n$. Then
  $\hom(\Xi^{\infty+n}\ast,\cat{C})
  \to\hom(\Xi^{\infty+n}\walkeq,\cat{C})$ is equivalent to
  $\hom(\Xi^{\infty+n+k}\ast,\Xi^{k}\cat{C})
  \to\hom(\Xi^{\infty+n+k}\walkeq,\Xi^{k}\cat{C})$ as $\Xi^{k}$ is in
  particular fully faithful, so by the above it is invertible since
  $\Xi^{k}\cat{C}$ is an $(\infty,\mathbb{Z})$-category and
  $n+k\geq0$.
\end{proof}

\section{Comparison with categorical spectra}
\label{sec:comp-with-categ-sp}

\subsection{Categorical spectra}
\label{sec:recoll-categ-spectra}

\begin{construct}
  The \textbf{reduced suspension} of a pointed
  $(\infty,\omega)$-category $\ast\xrightarrow{\characmap{C}}\cat{C}$
  is the pushout
  \begin{equation}
    \label{eq:red-susp-def}
    \begin{tikzcd}[column sep=large]
      \cat{D}_{1}=\Xi_{\biptd}\ast \arrow[d]
      \arrow[r,"\Xi_{\biptd}\characmap{C}"] \arrow[dr,phantom,very
      near
      end,"\ulcorner"description] & \Xi_{\biptd}\cat{C} \arrow[d] \\
      \ast \arrow[r] & \Sigma_{C}\cat{C}\text{.}
    \end{tikzcd}
  \end{equation}

  The functor $\Sigma\colon\infinfcats_{\ast}\to\infinfcats_{\ast}$
  admits a right-adjoint denoted $\Omega$: by a dualisation
  of~\cite[Proposition 5.2.5.1]{lurie09:_higher},
  $\Xi_{\biptd}\colon(\ast\downarrow\infinfcats)
  \to(\Xi_{\biptd}\ast\downarrow\infinfcats_{\ast,\ast})$ has a
  right-adjoint mapping $\cat{D}_{1}\to\cat{C}$ to
  $\ast\xrightarrow{\simeq}\Eta_{\biptd}\cat{D}_{1}\to\Eta_{\biptd}\cat{C}$,
  while the cobase-change $(-\amalg_{\cat{D}_{1}}\ast)$ has a
  right-adjoint given by precomposition along
  $\cat{D}_{1}\to\ast$. Concretely, it acts as
  \begin{equation}
    \label{eq:ptd-loop-def}
    \Omega\colon(\cat{C},C)
    \mapsto\Omega_{C}\cat{C}\coloneqq\hom_{\cat{C}}(C,C)\text{.}
  \end{equation}
\end{construct}

\begin{exmp}
  \label{exm:suspend-free-ptd}
  Let $\cat{C}$ be an $(\infty,\omega)$-category, and consider the
  freely pointed $(\infty,\omega)$-category
  $\cat{C}_{+}\simeq\cat{C}\amalg\ast$. Then
  \begin{equation}
    \label{eq:suspend-free-ptd}
    \Sigma\cat{C}_{+}
    \simeq\Xi_{\biptd}\cat{C}\lmamalg{\ast\amalg\ast}\ast
  \end{equation}
  as used in~\cite[Definition
  13.2.1]{lessard19:_spect_local_finit_z_group}.
\end{exmp}

\begin{construct}
  Thus far we have worked with $(\infty,\omega)$-categories seen only
  as Segal $\op{\Theta}$-$\infty$-groupoids. To relate our
  constructions to those of~\cite{stefanich21:_higher_quasic_sheav},
  we will need to use another point of view.


  By~\cite{goldthorpe23:_homot_theor_categ_univer_fixed}, we can see
  $(\infty,\omega)$-categories as $\infty$-categories enriched in
  $\infinfcats$. From this point of view, the desuspension operation
  takes a very simple form: it takes a pointed $\infinfcats$-enriched
  category $(\cat{C},C\in\obj\cat{C})$ to $\hom(C,C)$ (formalised
  in~\cite[Notation 13.1.4]{stefanich21:_higher_quasic_sheav} as the
  full sub-$\infinfcats$-category of $\cat{C}$ on the object $C$, seen
  as an algebra in $\infinfcats$ and forgetting the algebra structure
  bar the pointing at the unit).

  Finally, we recall that~\cite[Theorem 7.18 and
  Corollaries]{haugseng15:_rectif} provides an equivalence between
  these two points of view, and by~\cite[Lemma
  3.1.4]{masuda24:_algeb_categ_spect} the two definitions of reduced
  (de)suspenstion agree.
\end{construct}

\begin{definition}
  The $(\infty,1)$-category of \textbf{categorical spectra} is the
  limit
  \begin{equation}
    \label{eq:catl-spectra-def}
    \catsp\coloneqq\lim\bigl(\cdots\xrightarrow{\Omega}\infinfcats_{\ast}
    \xrightarrow{\Omega}\infinfcats_{\ast}
    \xrightarrow{\Omega}\infinfcats_{\ast}\bigr)\text{.}
  \end{equation}
\end{definition}


All that remains to do is to reduce the arguments about pointed
$(\infty,\omega)$-categories to arguments about sheaves. The
following result is standard, but we record the proof as we have not
seen it in the literature.

\begin{lemma}
  \label{lemma:ptd-prshvs}
  Let $\cat{C}$ be an $(\infty,1)$-category endowed with either a
  Grothendieck topology (so we can talk of sheaves on $\cat{C}$) or an
  algebraic pattern structure (so we can talk of presheaves respecting
  the Segal condition, also seen as ``sheaves'').  The
  $\infty$-categories of (pre)sheaves of pointed $\infty$-groupoids
  and of pointed (pre)sheaves on $\cat{C}$ are equivalent.
\end{lemma}

\begin{proof}
  As slice $\infty$-categories can be written as commas and hom
  functors preserve limits, we have
  \begin{equation}
    \label{eq:hom-commas-ptd-shves}
    \begin{split}
      \funcs{\op{\cat{C}}}{\infgrpds_{\ast}}
      &=\funcs{\op{\cat{C}}}{\ast\downarrow\infgrpds}\\
      &\simeq\funcs{\op{\cat{C}}}{\ast}\downarrow
      \funcs{\op{\cat{C}}}{\infgrpds}
      \simeq\ast\downarrow\funcs{\op{\cat{C}}}{\infgrpds}\text{.}
    \end{split}
  \end{equation}

  The passage to sheaves follows from the fact (proved
  in~\cite[Proposition 1.2.13.8]{lurie09:_higher}) that limits in
  undercategories are computed in the projection.
\end{proof}

\begin{lemma}
  \label{lemma:desusp-is-desusp}
  Under the equivalence
  $\infinfcats_{\ast}
  \simeq\funcs{\op{\Theta}}{\infgrpds_{\ast}}^{\text{Seg,univ-cplt}}$,
  the reduced desupension functor $\Omega$ corresponds to
  $\Xi^{\ast}$.
\end{lemma}

\begin{proof}
  The left-adjoint $\Xi_{!}$ to $\Xi^{\ast}$, by its construction as a
  Yoneda extension, is uniquely characterised by its restriction along
  $\Theta\xrightarrow{\yo}\funcs{\op{\Theta}}{\infgrpds}
  \xrightarrow{(-)_{+}}\funcs{\op{\Theta}}{\infgrpds_{\ast}}$
  and the fact that it preserves colimits. One readily sees, thanks
  to~\cref{exm:suspend-free-ptd}, that the restriction of $\Xi_{!}$ to
  $\Theta$ is the reduced suspension of globular sums, and since the
  functor $\Sigma$ is a left-adjoint it preserves colimits, so it has
  to coincide with $\Xi_{!}$.
\end{proof}

\begin{thm}
  \label{thm:main-thm}
  There is an equivalence of $\infty$-categories
  $\infzcats_{\ast}\simeq\catsp$.
\end{thm}

\begin{proof}
  By~\cref{lemma:ptd-prshvs}, pointed $(\infty,\mathbb{Z})$-categories
  are (univalent-complete Segal) presheaves of pointed
  $\infty$-groupoids on $\Theta_{\mathbb{Z}}$. One sees easily, as
  before in the proof of~\cref{corlr:flaggd-zcats-lim-suspensions},
  that the Segal and univalence-completeness conditions distribute to
  each of the factors in the diagram
  \begin{equation}
    \label{eq:1}
    \begin{tikzcd}
      \cdots\xrightarrow{\Xi^{\ast}}\funcs{\op{\Theta}}{\infgrpds_{\ast}}
      \xrightarrow{\Xi^{\ast}}\funcs{\op{\Theta}}{\infgrpds_{\ast}}
      \xrightarrow{\Xi^{\ast}}\funcs{\op{\Theta}}{\infgrpds_{\ast}}\text{.}
    \end{tikzcd}
  \end{equation}

  Then, by~\cref{lemma:desusp-is-desusp}, this diagram coincides
  with~\cref{eq:catl-spectra-def} defining categorical spectra.
\end{proof}

\subsection{Monoidal structures}
\label{sec:monoidal-structures}

A pointed object is, by~\cite[Remark 2.1.3.10]{lurie17:_higher_algeb},
the same thing as an algebra over the $\infty$-operad
$\oprd{E}_{0}$. Thus the definition of the $\infty$-category
categorical spectra can be interpreted as the limit of
$\cdots\xrightarrow{\Omega}
\algs{\oprd{E}_{0}}(\infinfcats)\xrightarrow{\Omega}
\algs{\oprd{E}_{0}}(\infinfcats)$, which by~\cref{thm:main-thm}
becomes equivalent to $\algs{\oprd{E}_{0}}(\infzcats)$.

We will now show, using a result
of~\cite{stefanich21:_higher_quasic_sheav}, that this generalises to
$\oprd{E}_{k}$-algebras for any value of $k$.

For simplicity of notation, we will implicitly make use of the
identification
$\algs{\oprd{E}_{k}}(\cat{C})\simeq\algs{\oprd{E}_{k}}(\cat{C}_{\ast})$
(\emph{cf.}~\cite[Notation 5.2.6.11]{lurie17:_higher_algeb}) for any
cartesian monoidal $(\infty,1)$-category $\cat{C}$. We will also use
the additivity theorem to identify $\oprd{E}_{1}$-algebras in
$\oprd{E}_{k}$-algebras with $\oprd{E}_{k+1}$-algebras.

\begin{pros}
  For any $0\leq k\leq\infty$, there is an equivalence of
  $(\infty,1)$-categories
  \begin{equation}
    \algs{\oprd{E}_{k}}(\infzcats)
    \simeq\lim\bigl(\cdots\xrightarrow{\Omega}
    \algs{\oprd{E}_{k}}(\infinfcats)\xrightarrow{\Omega}
    \algs{\oprd{E}_{k}}(\infinfcats)\bigr)\text{.}
  \end{equation}
\end{pros}

\begin{proof}
  Recall from~\cite[Remark 3.12]{chu21:_homot_segal} that (the
  $(\infty,1)$-category of operators of) any $(\infty,1)$-operad
  $\oprd{O}$ is canonically endowed with a structure of algebraic
  pattern, such that Segal $\oprd{O}$-objects in a complete
  $(\infty,1)$-category $\cat{C}$ coincide with $\oprd{O}$-algebras in
  the cartesian monoidal $\infty$-category $\cat{C}$. As an upshot,
  and by uncurrying, $\oprd{E}_{k}$-monoidal
  $(\infty,\mathbb{Z})$-categories are Segal objects in $\infgrpds$
  for the product pattern
  $\oprd{E}_{k}\times\op{\Theta}_{\mathbb{Z}}$.

  Since limits commute with filtered (in particular, sequential)
  colimits, we also have from the definition of
  $\op{\Theta}_{\mathbb{Z}}$ that
  $\oprd{E}_{k}\times\op{\Theta}_{\mathbb{Z}}
  \simeq\colim\bigl(\oprd{E}_{k}\times\op{\Theta}
  \xrightarrow{\oprd{E}_{k}\times\op{\Xi}}
  \oprd{E}_{k}\times\op{\Theta}
  \xrightarrow{\oprd{E}_{k}\times\op{\Xi}}\cdots\bigr)$, so that
  passing to Segal $\infty$-groupoids produces the sought-after
  equivalence.
\end{proof}

\begin{remark}
  Specialising to $k=0$, this gives another proof
  of~\cref{thm:main-thm}.
\end{remark}

\begin{pros}[{\cite[Proposition
    13.4.4]{stefanich21:_higher_quasic_sheav}}]
  \label{pros:catsp-mondl-indep}
  For any $0\leq k\leq\infty$, the forgetful functors
  $\algs{\oprd{E}_{k}}(\infinfcats)\to\algs{\oprd{E}_{0}}(\infinfcats)$
  induce an equivalence
  \begin{equation}
    \label{eq:catsp-lim-ek-mondl-catsp}
    \catsp\simeq\lim\bigl(\cdots\xrightarrow{\Omega}
    \algs{\oprd{E}_{k}}(\infinfcats)\xrightarrow{\Omega}
    \algs{\oprd{E}_{k}}(\infinfcats)\bigr)\text{.}
  \end{equation}
\end{pros}

\begin{proof}
  \cite{stefanich21:_higher_quasic_sheav} proves the case $k=\infty$,
  but we show that a mild adaptation gives the result for any $k$.

  Consider the diagram
  \begin{equation}
    \label{eq:6}
    \begin{tikzcd}
      \vdots \arrow[d,equals] & \vdots \arrow[d,equals] & \vdots
      \arrow[d] &
      \arrow[l,phantom,"\cdots"description] \\
      \algs{\oprd{E}_{k}}(\infinfcats) \arrow[d,equals] &
      \algs{\oprd{E}_{k}}(\infinfcats) \arrow[d] \arrow[l,"\Omega"'] &
      \algs{\oprd{E}_{k-1}}(\infinfcats) \arrow[d] \arrow[l,"\Omega"'] &
      \arrow[l,phantom,"\cdots"description] \\
      \algs{\oprd{E}_{k}}(\infinfcats) \arrow[d] &
      \algs{\oprd{E}_{k-1}}(\infinfcats) \arrow[d] \arrow[l,"\Omega"] &
      \algs{\oprd{E}_{k-2}}(\infinfcats) \arrow[d] \arrow[l,"\Omega"] &
      \arrow[l,phantom,"\cdots"description] \\
      \vdots \arrow[d] & \vdots \arrow[d] & \vdots \arrow[d] &
      \arrow[l,phantom,"\cdots"description] \\
      \algs{\oprd{E}_{1}}(\infinfcats) \arrow[d] &
      \algs{\oprd{E}_{0}}(\infinfcats) \arrow[d] \arrow[l,"\Omega"'])&
      \algs{\oprd{E}_{0}}(\infinfcats) \arrow[d] \arrow[l,"\Omega"'] &
      \arrow[l,phantom,"\cdots"description] \\
      \algs{\oprd{E}_{0}}(\infinfcats) &
      \algs{\oprd{E}_{0}}(\infinfcats) \arrow[l,"\Omega"])&
      \algs{\oprd{E}_{0}}(\infinfcats) \arrow[l,"\Omega"] &
      \arrow[l,phantom,"\cdots"description]
    \end{tikzcd}
  \end{equation}

  The limit of each $n$th row is equivalent to $\catsp$ (as the limit
  of the sequence starting after the $n$th term), so that the limit of
  the induced column again computes $\catsp$.

  On the other hand, the limit of each $p$th column is
  $\algs{\oprd{E}_{k}}(\infinfcats)$, either, if $k$ is finite,
  because the sequence starting after the $(k+p)$th term is constant
  with value $\algs{\oprd{E}_{k}}(\infinfcats)$, or if $k$ is infinite
  because as reminded in the proof of~\cite[Proposition
  13.4.4]{stefanich21:_higher_quasic_sheav}
  $\algs{\oprd{E}_{\infty}}(\infinfcats)
  \simeq\lim_{p}\algs{\oprd{E}_{p}}(\infinfcats)$. Thus the limit of
  the induced row is
  $\lim\bigl(\cdots\xrightarrow{\Omega}
  \algs{\oprd{E}_{k}}(\infinfcats)\xrightarrow{\Omega}
  \algs{\oprd{E}_{k}}(\infinfcats)\bigr)$.
\end{proof}

\begin{corlr}
  For any $0\leq k<n\leq\infty$, the forgetful functor
  $\algs{\oprd{E}_{n}}\bigl(\infzcats\bigr)
  \to\algs{\oprd{E}_{k}}\bigl(\infzcats\bigr)$ is an
  equivalence.\qed{}
\end{corlr}

This shows that the (unpointed) stabilisation process giving rise to
$(\infty,\mathbb{Z})$-categories already allows for a very strong form
of delooping: as soon as an $(\infty,\mathbb{Z})$-category is equipped
with a pointing, it can be fully delooped to an
$\oprd{E}_{\infty}$-monoidal $(\infty,\mathbb{Z})$-category.

\subsection{Comparison of cells}
\label{sec:comparison-cells}

\begin{definition}
  Let $\cat{X}$ be a categorical spectrum, corresponding to a sequence
  of pointed $(\infty,\omega)$-categories
  $(\cat{X}_{n})_{n\in\mathbb{N}}$ with equivalences
  $\Omega\cat{X}_{n+1}\simeq\cat{X}_{n}$. For any $k\in\mathbb{N}$,
  the $\infty$-groupoid of \textbf{$k$-cells} in $\cat{X}$ is the
  sequential colimit
  \begin{equation}
    \label{eq:def-cells-catsp}
    \cells_{k}\cat{X}\coloneqq
    \colim\bigl(\hom(\cat{D}_{k},\cat{X}_{0})
    \to\hom(\cat{D}_{k+1},\cat{X}_{1})
    \to\hom(\cat{D}_{k+2},\cat{X}_{2})\to\cdots\bigr)\text{.}
  \end{equation}
\end{definition}

As pointed out in~\cite[Remark
13.2.14]{stefanich21:_higher_quasic_sheav}, we have
$\cells_{k}\Omega\cat{X}\simeq\cells_{k+1}\cat{X}$, so that for $k<0$
we may also define $\cells_{k}\cat{X}=\cells_{0}\Sigma^{-k}\cat{X}$.

\begin{lemma}
  \label{lemma:z-cells-are-cells}
  For any $(\infty,\mathbb{Z})$-category $\cat{C}$ corresponding to
  $(\cat{C}_{k})_{k\in\mathbb{N}}$, and for any $k\in\mathbb{N}$ the
  colimit
  \begin{equation}
    \label{eq:def-cells-zcat}
    \colim\bigl(\hom(\cat{D}_{k},\cat{C}_{0})
    \to\hom(\cat{D}_{k+1},\cat{C}_{1})
    \to\hom(\cat{D}_{k+2},\cat{C}_{2})\to\cdots\bigr)
  \end{equation}
  is equivalent to $\hom(\Xi^{\infty}\cat{D}_{k},\cat{C})$.
\end{lemma}

\begin{proof}
  For any $i\in\mathbb{N}$, we have
  \begin{equation}
    \label{eq:4}
    \hom(\cat{D}_{k+i},\cat{C_{i}})
    \simeq\hom(\Xi^{i}\cat{D}_{k},\Eta^{\infty-i}\cat{C})
    \simeq\hom(\Xi^{\infty-i+i}\cat{D}_{k},\cat{C})\text{,}
  \end{equation}
  so all terms in the colimit are equivalent to
  $\hom(\Xi^{\infty}\cat{D}_{k},\cat{C})$. Because we are taking a
  colimit in the $(\infty,1)$-category of $\infty$-groupoids, the
  colimit of a constant diagram is the copower with the
  groupoidification of the indexing shape category, which in our case
  is $\mathbb{N}$ so contractible.
\end{proof}

In particular, positive-dimensional cells of an
$(\infty,\mathbb{Z})$-category are simply cells in its infinite
desuspension, whereas for categorical spectra it fails to be true that
positive-dimensional cells are cells in the infinite delooping.

\begin{thm}
  \label{thm:cells}
  Let $\cat{X}$ be a categorical spectrum, and let
  $\ast\to\kappa\cat{X}$ be the corresponding pointed
  $(\infty,\mathbb{Z})$-category. For any $k\in\mathbb{Z}$ there is an
  equivalence
  $\cells_{k}\cat{X}\simeq\hom(\Xi^{\infty}\cat{D}_{k},\kappa\cat{X})$.
\end{thm}

\begin{proof}
  By~\cref{lemma:z-cells-are-cells}, we have
  $\hom(\Xi^{\infty}\cat{D}_{k},\kappa\cat{X})
  \simeq\colim_{i\in\mathbb{N}}\bigl(
  \hom(\cat{D}_{k+i},(\kappa\cat{X})_{i})\bigl)$, and as the
  equivalence $\kappa\colon\catsp\xrightarrow{\simeq}\infzcats_{\ast}$
  is compatible with the infinite deloopings this is equivalent to
  $\cells_{k}\cat{X}$.
\end{proof}

\begin{construct}
  The projection $\cat{D}_{1}\to\cat{D}_{0}$ induces a degeneracy map
  $\cells_{k-1}\cat{X}\to\cells_{k}\cat{X}$ for any $k\in\mathbb{Z}$
  and any $\cat{X}\in\catsp$. A $k$-cell in a categorical spectrum
  $\cat{X}$ is said to be \textbf{invertible} if it lies in the image
  of this degeneracy. More concretely, \cite[Proposition
  13.2.19]{stefanich21:_higher_quasic_sheav} shows that a $k$-cell
  $\varphi$ of $\cat{X}$ is invertible if and only if some (and thus,
  any) representative
  $\overline{\varphi}\in\hom(\cat{D}_{k+i},\cat{X}_{i})$ defines an
  invertible $(k+i)$-cell of the $(\infty,\omega)$-category
  $\cat{X}_{i}$.
  
  By~\cite[Remark 13.3.2]{stefanich21:_higher_quasic_sheav},
  categorical spectra all of whose cells are invertible coincide with
  spectra.
\end{construct}

\begin{corlr}[{\cite[Corollary
    14.3.7]{lessard19:_spect_local_finit_z_group}}]
  \label{corlr:lessard-grpd-spt}
  The equivalence $\catsp\simeq\infzcats_{\ast}$ restricts to an
  equivalence between the $\infty$-categories of pointed groupoidal
  $(\infty,\mathbb{Z})$-categories and of spectra.
\end{corlr}

\begin{proof}
  By univalence, a cell of an $(\infty,\mathbb{Z})$-category is
  invertible if and only if it lies in the image of a degeneracy
  map. Thus the notion of invertibility is preserved by the
  equivalence $\catsp\simeq\infzcats_{\ast}$, so that restricting it
  to groupoidal objects produces the sought-after equivalence.
\end{proof}

\printbibliography

~

{\footnotesize \textsc{David Kern, KTH Royal Institute of Technology,
    Department of Mathematics, SE-100 44 Stockholm, Sweden}

  \textit{Email address}:
  \href{mailto:dkern@kth.se}{\texttt{dkern@kth.se}}

  \emph{URL}:
  \href{https://dskern.github.io/}{\texttt{https://dskern.github.io/}}}

\end{document}